\UseRawInputEncoding
\documentclass[12pt,reqno]{amsart}
\usepackage{amssymb,amsmath,graphicx,
	amsfonts,euscript}
\usepackage{color}
\usepackage{mathrsfs}
\theoremstyle{plain}

 \newtheorem{thm}{Theorem}[section]
 \newtheorem{cor}[thm]{Corollary}
 \newtheorem{lem}[thm]{Lemma}
 
 \theoremstyle{definition}
 \newtheorem{defn}[thm]{Definition}
 \theoremstyle{remark}
 \newtheorem{rem}[thm]{Remark}
 \numberwithin{equation}{section}

\begin{document}

\title[steady solution for the fractal Burgers equation]{Uniqueness and stability of steady-state solution with finite energy to the fractal Burgers equation}

\author[ Fei Xu, Yong Zhang, Fengquan Li]{ Fei Xu, Yong Zhang, Fengquan Li}

\address[Yong Zhang]{ School of Mathematical Sciences, Dalian University
of Technology, Dalian, 116024, China} \email{18842629891@163.com}

\address[Fei Xu]{School of Mathematical Sciences, Dalian University
of Technology, Dalian, 116024, China}

\address[Fengquan Li]{School of Mathematical Sciences, Dalian University
of Technology, Dalian, 116024, China}

\begin{abstract}
The paper is concerned with the steady-state Burgers equation of fractional dissipation on the real line. We first prove the global existence of viscosity weak solutions to the fractal Burgers equation (\ref{e11}) driven by the external force. Then the existence and uniqueness of solution with finite $H^{\frac{\alpha}{2}}$ energy to the steady-state equation (\ref{e12}) are established by estimating the decay of fractal Burgers' solutions. Furthermore, we show that the unique steady-state solution is nonlinearly stable, which means any viscosity weak solution of (\ref{e11}), starting close to the steady-state solution, will return to the steady state as $t\rightarrow\infty$.
	\end{abstract}
\maketitle
%\noindent {\bf Key Words:}
%{Global solutions;  Oldroyd-B model; inviscid limit;  Besov space}

%\noindent {\bf Mathematics Subject Classification (2010)} {76A10; 76D03 }
\section{Introduction and main results}
In this paper, we consider the forced Burgers equation of fractional order
\begin{equation}\label{e11}
u_{t}+uu_{x}+\Lambda^{\alpha}u=f,\quad (x,t)\in R\times(0,T),
\end{equation}
 where $\alpha\in (1, \frac{3}{2})$ and $\Lambda^{\alpha}u=(-\partial_{xx})^{\frac{\alpha}{2}}u$ is defined by the Fourier transform
$$
\widehat{\Lambda^{\alpha}u}(\xi)=|\xi|^{\alpha}\hat{u}(\xi).
$$
If $u_{t}\equiv0$, then equation (\ref{e11}) would reduce to the steady-state equation
\begin{equation}\label{e12}
\frac{1}{2}(U^{2})_{x}+\Lambda^{\alpha}U=f,
\end{equation}
where $U$ and $f$ are independent of time variable $t$.

Burgers equation can be viewed as the simplest partial differential equation to model the Euler and Navier-Stokes equations's nonlinearity. Since the studies by Burgers in the 1940s, there are many important investigations on (\ref{e11}) without external force, i.e., $f=0$.  If removing the dissipative term $\Lambda^{\alpha}u$, the equation is perhaps the most
basic example to lead to shocks. If $\alpha = 2$, it provides an accessible
model for studying the interaction between nonlinear and dissipative phenomena. Besides, in recent years, there has been a great deal of interest in using the fractional dissipation to describe diverse physical phenomena, such as anomalous diffusion and quasi-geostrophic flows, turbulence and water waves and molecular dynamics (see \cite{13,16,17,18} and the references therein). Burgers equation with fractional dissipation, i.e. $0<\alpha<2$, has received an extensive amount of attention such as in \cite{1,2,20}. In the supercritical dissipative case ($0<\alpha<1$), the equation is locally well-posed and its solution develops gradient blow-up in finite time. In the critical dissipative case ($\alpha=1$) and subcritical dissipative case ($1<\alpha<2$), such singularity does not appear so that solution always exists globally in time. In additin, the results on the global regularizing effects in the subcritical case and the non-uniqueness of weak solutions in the supercritical case were also established in \cite{11} and \cite{12}, respectively. The author in \cite{3} obtained the analyticity and large time behavior for Burgers equation with critical dissipation.

In the case of $f=f(t,x)\neq 0$ and $\alpha=2$, equation (\ref{e11}) is reduced to the non-autonomously forced Burgers equation. The authors in \cite{4} investigated its Dirichlet and periodic boundary value problems and obtained the unique $H^{1}$ bounded trajectory, which attracts all trajectories both in pullback and forward sense. In \cite{5}, the existence of the time-periodic solution was established, where the authors also proved that this time-periodic solution was unique and asymptotically stable in the $H^{1}$ sense. For the case of  $1<\alpha<\frac{3}{2}$, we recently obtained the existence and uniqueness of time-periodic solution to (\ref{e11}) in \cite{19}. In this paper, we mainly consider the steady-state equation (\ref{e12}) with $1<\alpha<\frac{3}{2}$ and $f=f(x)\neq 0$. It's worth noting that the index $\alpha=\frac{3}{2}$ may be a limit by using the energy method. It's expected to extend this result to the case of $1< \alpha< 2$ in future investigation.

The paper is organized as follows. In section 2, we collect some useful lemmas, which would be used throughout the paper. In section 3, we prove the global existence of viscosity weak solutions to (\ref{e11}) by using Galerkin method, where the compactness (Aubin-Lions) lemma plays a key role. In section 4, we establish
the existence and uniqueness of solution to the steady-state equation (\ref{e12}), which is a novel promotion of Bjorland and Schonbek \cite{6} for Navier-Stokes equation in $R^{3}$ and Dai \cite{7} for quasi-geostrophic equation in $R^{2}$. In comparison, since
the velocity $u$ is not divergence free in $R$, we no-longer have energy-like inequality,
which naturally gives control of the spatial  $L^{2}$ norm of solutions.
 Here the nontrivial part is to establish the boundedness of $\|U\|_{L^{2}}$. Although the Poincar\'{e} type inequality is not available in the whole space $R$, we first establish the decay rate of the evolutionary solution in $R$. In contrast to the Navier-Stokes equation and quasi-geostrophic equation, the equation (\ref{e11}) has a weaker dissipative term $\Lambda^{\alpha}u$ with $\alpha \in (1, \frac{3}{2})$, which presents a serious obstacle
 to establish the decay rate. Once the finite energy estimate is established, the smallness assumption on the force yields the uniqueness of the steady-state solution. The section 5 is devoted to establish that the steady-state solution $U$ is nonlinearly stable in the sense: let $\theta\in L^{2}$ be a perturbation and $u(x,t)$ be a viscosity weak solution of the forced Burgers equation (\ref{e11}) with initial data $u_{0}=\theta+U$, then $\lim_{t\rightarrow\infty}\|u(t)-U\|_{L^{2}}=0$. The boundedness of a viscosity solution in $L^{\infty}(R)$ is essential in the stability analysis, which would been given in subsection 5.2.

Now let's state our main results in the following. Throughout this paper, we use $C$ to denote the universal constant and use $C(s)$ to denote the positive constant depending on $s$.

\begin{thm}(Main theorem)\label{thm1.4}
For any given $0<\epsilon<1$, assume $\alpha\in(1,\frac{3}{2+\epsilon})$ and $f\in X=\dot{H}^{-\frac{\alpha}{2}}(R)\cap H^{\frac{\alpha}{2}}(R)$ satisfies
$$
\widehat{f}(\xi)=0, ~~as~~ \xi\in\{~|\xi|<\rho~~for ~some~\rho>0\},
$$
then there exist a constant $C(\alpha, \epsilon)$ such that if $\|f\|_{X}\leq C(\alpha, \epsilon)$,
 equation (\ref{e12}) admits a unique weak solution $U\in H^{\frac{\alpha}{2}}(R)$ in the sense that
\begin{equation}\label{e14}
-\frac{1}{2}\langle U^{2}, \psi_{x}\rangle+\langle \Lambda^{\frac{\alpha}{2}}U, \Lambda^{\frac{\alpha}{2}}\psi\rangle=\langle f, \psi\rangle,~~for~~\psi\in C_{0}^{\infty}(R),
\end{equation}
with the finite $H^{\frac{\alpha}{2}}$ norm
\begin{equation}\label{e15}
\|U\|_{H^{\frac{\alpha}{2}}(R)}\leq C\epsilon^{-1}\|f\|_{X}.
\end{equation}
\end{thm}
\begin{rem}
This theorem implies that the weak solution of (\ref{e12}), which belongs to the class with bounded $H^{\frac{\alpha}{2}}$ energy, is unique.
\end{rem}
\begin{thm}\label{thm1.5}
If $f$ satisfy the assumptions in Theorem \ref{thm1.4}, the unique weak solution $U$ of (\ref{e12}) is nonlinearly stable in the following sense:
 let $\theta\in L^{2}(R)$ and $u(t,x)$ be a viscosity weak solution of (\ref{e11}) obtained in Corollary \ref{cor3.4} with initial date $u(0,x)=U+\theta$, then for any $T>0$, there hold
\begin{equation}\label{e1.6}
 u\in L^{\infty}((0,T)\times R)
 \end{equation}
 and
 \begin{equation}\label{e1.7}
 \lim_{t\rightarrow \infty}\|u(t)-U\|_{L^{2}}=0.
 \end{equation}
\end{thm}

\section{Preliminary}
In this section, we give some useful lemmas we need in the following. We mention that $u\in H^{s}(R)$ (or $u\in \dot{H}^{s}(R)$) means that the tempered distribution $u$ satisfies $\widehat{u}\in L_{loc}^{2}(R)$
and
\begin{equation*}
\|f\|_{H^{s}}^{2}=\int_{R}(1+|\xi|^{2})^{s}|\hat{u}|^{2}d\xi<\infty ~~ (or~ \|f\|_{H^{s}}^{2}=\int_{R}|\xi|^{2s}|\hat{u}|^{2}d\xi<\infty).
\end{equation*}
\begin{lem}\label{lem2.1}
Let $2\leq p <\infty$ and $\delta=\frac{1}{2}-\frac{1}{p}$, then there exists a constant $C(p)$ such that
\begin{equation}\label{e2.1}
\|f\|_{L^{p}}\leq C(p)\|\Lambda^{\delta}f\|_{L^{2}}.
\end{equation}
\end{lem}
\begin{lem}\label{lem2.2}
The interpolation lemma
\begin{equation}\label{e2.2}
\|f\|_{L^{p}}\leq C(p)\|f\|_{L^{2}}^{a}\|\Lambda^{\sigma}f\|_{L^{2}}^{1-a}\leq C(\|f\|_{L^{2}}+\|\Lambda^{\sigma}f\|_{L^{2}})
\end{equation}
with $\frac{1}{p}=\frac{a}{2}+(1-a)(\frac{1}{2}-\sigma)$ and $0\leq a \leq 1.$
\end{lem}
\begin{lem}(see \cite{21})\label{lem2.3}
The commutator estimates
\begin{equation}\label{e2.3}
\|\Lambda^{s}(fg)\|_{L^{p}}\leq C(\|\Lambda^{s}f\|_{L^{p_{1}}}\|g\|_{L^{q_{1}}}+\|f\|_{L^{p_{2}}}\|\Lambda^{s}g\|_{L^{q_{2}}}),
\end{equation}
for $s>0$ and $\frac{1}{p}=\frac{1}{p_{1}}+\frac{1}{q_{1}}=\frac{1}{p_{2}}+\frac{1}{q_{2}}$.
\end{lem}
\begin{lem}(Aubin-Lions lemma in \cite{10})\label{lem2.4}
 Let $X\subset Y\subset Z$ be reflexive Banach spaces with the imbedding $X\subset Y$ being compact. Assume that \\
 (1) For any $\delta>0$, there is $C_{\delta}>0$ such that
 $$
\|x\|_{Y}\leq\delta\|x\|_{X}+C_{\delta}\|x\|_{Z},~~~ \forall x\in X.
 $$
 (2) If $0<T<\infty$, $p\geq1, q>1$ and $g_{j}\in L^{p}(0,T;X), j\in N$, satisfy
 $$
 \int^{T}_{0}\|g_{j}(t)\|_{X}^{p}dt+\int^{T}_{0}\|\frac{d}{dt}g_{j}(t)\|_{Z}^{q}dt\leq C(T)
 $$
 for some constant $C(T)<\infty$, then $\{g_{j}\}_{j}$ is relatively compact in $L^{p}(0,T;Y)$.
 \end{lem}

\section{The existence of viscosity weak solution to forced Burgers equation}
This section is devoted to establish the existence of viscosity weak solution to forced Burgers equation. Let's first recall the definition of the weak solution and viscosity weak solution.
 \begin{defn}\label{def3.1}
 A weak solution to (\ref{e11}) is a function $u\in C_{w}([0, T]; L^{2}(R))$ satisfying, for any $\psi \in C_{0}^{\infty}(R\times(0,T))$,
\begin{equation}
\begin{aligned}
~&-\int^{T}_{0}\langle u, \psi_{t}\rangle dt-\frac{1}{2}\int^{T}_{0}\langle(u)^{2},\psi_{x}\rangle dt+\int^{T}_{0}\langle\Lambda^{\frac{\alpha}{2}}u,
\Lambda^{\frac{\alpha}{2}}\psi\rangle dt \nonumber\\
&=\langle u_{0},\psi(x,0)\rangle+\int^{T}_{0}\langle f,\psi\rangle dt
 \end{aligned}
\end{equation}
\end{defn}
\begin{defn}\label{def3.2}
A weak solution of (\ref{e11})
will be called a viscosity weak solution with initial data $u_{0}\in L^{2}(R)$, if it is the weak limit of a sequence of solutions of the problem
\begin{equation}\label{e30}
u^{\varepsilon}_{t}+u^{\varepsilon}u^{\varepsilon}_{x}+\Lambda^{\alpha}u^{\varepsilon}=f(x)+\varepsilon u^{\varepsilon}_{xx},\quad as~~\varepsilon\rightarrow 0,
\end{equation}
with $u^{\varepsilon}(x,0)=u_{0}$.
\end{defn}
\begin{lem}\label{lem3.3}
If $u_{0}\in L^{2}(R)$ and $f\in \dot{H}^{-\frac{\alpha}{2}}(R)\cap L^{2}(R)$, then there exists a global weak solution $u^{\varepsilon}\in L^{\infty}([0,T),L^{2}(R))\cap L^{2}([0,T),\dot{H}^{1}(R))$ to (\ref{e30}), which satisfies for any given $T>0$,
\begin{equation}
\begin{aligned}
~&-\int^{T}_{0}\langle u^{\varepsilon}, \psi_{t}\rangle dt-\frac{1}{2}\int^{T}_{0}\langle(u^{\varepsilon})^{2},\psi_{x}\rangle dt+\int^{T}_{0}\langle\Lambda^{\frac{\alpha}{2}}u^{\varepsilon},
\Lambda^{\frac{\alpha}{2}}\psi\rangle dt \nonumber\\
&=\langle u_{0},\psi(x,0)\rangle+\int^{T}_{0}\langle f,\psi\rangle dt-\varepsilon\int^{T}_{0}\langle u^{\varepsilon}_{x},\psi_{x}\rangle dt,~~for~~\psi\in C_{c}^{\infty}([0,T)\times R).
 \end{aligned}
\end{equation}
Moreover, there holds
$$
\sup_{t\geq 0}\|u^{\varepsilon}(t)\|_{L^{2}}^{2}+\int^{T}_{0}\|\Lambda^{\frac{\alpha}{2}} u^{\varepsilon}(t)\|_{L^{2}}^{2}dt+\varepsilon\int^{T}_{0}\|u^{\varepsilon}_{x}(t)\|_{L^{2}}^{2}dt\leq \widetilde{C}(T),
$$
where $\widetilde{C}(T)=\|u_{0}\|_{L^{2}}^{2}+\frac{T}{2}\|\Lambda^{-\frac{\alpha}{2}}f\|_{L^{2}}^{2}$.
\end{lem}

\begin{proof}
We mainly use Galerkin method to project the equation (\ref{e30}) into a finite dimensional functional space to find a global approximated solution. The Aubin-Lions lemma is applied to conclude the compactness, which would ensure the approximated solutions converge to a weak solution to (\ref{e30}). In the following, we only show an energy estimate in (\ref{e3.3}) and a key compact result in (\ref{e3.4}), and other parts are standard arguments.

Firstly let's consider the following Galerkin approximation equation
\begin{equation}\label{e3.1}
\left\{ \begin{array}{ll}
\partial_{t}u^{\varepsilon}_{n}+u^{\varepsilon}_{n}\partial_{x}u^{\varepsilon}_{n}+\Lambda^{\alpha}u^{\varepsilon}_{n}=f(x)+\varepsilon \partial_{xx}u^{\varepsilon}_{n},  \\
u^{\varepsilon}_{n}(x,0)=u_{0}.
 \end{array} \right.
\end{equation}
Assume that the solutions of (\ref{e3.1}) can be written as
$$
u^{\varepsilon}_{n}(x,t)=\sum_{j=1}^{n}c^{n}_{j}(t)e_{j}(x),
$$
where $\{e_{j}\}_{j=1}$ are a family of orthonormal basis of $H^{1}(R)$. Taking the inner product on both sides of (\ref{e3.1}) in $L^{2}(R)$ with $e_{k}$ for $k=1,2,3,\cdot \cdot \cdot,n$, we can obtain
\begin{equation}
\begin{aligned}
~&\sum_{j=1}^{n}\langle\partial_{t}c^{n}_{j}(t)e_{j},e_{k}\rangle+\sum_{i,j=1}^{n}\langle c^{n}_{j}(t)e_{j}c^{n}_{i}(t)\partial_{x}e_{i},e_{k}\rangle+\sum_{j=1}^{n}\langle c^{n}_{j}(t)\Lambda^{\alpha} e_{j},e_{k}\rangle \nonumber\\
&=\langle f(x),e_{k}\rangle+\varepsilon\sum_{j=1}^{n}\langle\partial_{t}c^{n}_{j}(t)\partial_{xx}e_{j},e_{k}\rangle.
 \end{aligned}
\end{equation}
Since $\{e_{j}(x)\}$ are also orthonormal in $L^{2}(R)$, it follows that
\begin{equation}\label{e3.2}
\partial_{t}c^{n}_{k}(t)=-\lambda_{k}^{\frac{\alpha}{2}}c^{n}_{k}(t)+\varepsilon\lambda_{k}c^{n}_{k}(t)+\langle F(t,x),e_{k}\rangle,~~for~~k=1,2,...,n,
\end{equation}
where $F(t,x)=f(x)-\sum_{i,j=1}^{n}c^{n}_{j}(t)e_{j}c^{n}_{i}(t)\partial_{x}e_{i}$ and $\lambda_{k}$ is the $k$th eigenvalue of operator $-\partial_{xx}$. It's easy to see that the right hand of (\ref{e3.2}) is locally Lipschitz on $c^{n}(t)=(c^{n}_{1}(t),c^{n}_{2}(t),...,c^{n}_{n}(t))^{T}$, then the classical theory of ODEs yields the existence and uniqueness of solutions $\{c^{n}_{k}(t)\}_{k=1,2,...,n}$ for $t\in[0,T_{n})$, hence $u^{\varepsilon}_{n}(x,t)$ for $t\in[0,T_{n})$.

Now we show this solution $u^{\varepsilon}_{n}(x,t)$ is global in time by establishing the following uniform estimate (\ref{e3.3}). Taking $L^{2}$ inner product on (\ref{e3.1}) with $u^{\varepsilon}_{n}(x,t)$, we obtain
$$
\frac{1}{2}\frac{d}{dt}\|u^{\varepsilon}_{n}\|_{L^{2}}^{2}+\|\Lambda^{\frac{\alpha}{2}}u^{\varepsilon}_{n}\|_{L^{2}}^{2}+
\varepsilon\|\partial_{x}u^{\varepsilon}_{n}\|_{L^{2}}^{2}\leq\frac{1}{2}\|\Lambda^{\frac{\alpha}{2}}u^{\varepsilon}_{n}\|_{L^{2}}^{2}+2
\|\Lambda^{-\frac{\alpha}{2}}f\|_{L^{2}}^{2},
$$
that is to say,
$$
\frac{d}{dt}\|u^{\varepsilon}_{n}\|_{L^{2}}^{2}+\|\Lambda^{\frac{\alpha}{2}}u^{\varepsilon}_{n}\|_{L^{2}}^{2}+
2\varepsilon\|\partial_{x}u^{\varepsilon}_{n}\|_{L^{2}}^{2}\leq4\|\Lambda^{-\frac{\alpha}{2}}f\|_{L^{2}}^{2}.
$$
For any given $T$, integrating the above equality over $(0,T)$ gives
 \begin{equation}\label{e3.3}
 \sup_{t\geq0}\|u^{\varepsilon}_{n}\|_{L^{2}}^{2}+\int^{T}_{0}\|\Lambda^{\frac{\alpha}{2}}u^{\varepsilon}_{n}\|_{L^{2}}^{2}dt+2\varepsilon\int^{T}_{0}
 \|\partial_{x}u^{\varepsilon}_{n}\|_{L^{2}}^{2}dt\leq \widetilde{C}(T),
 \end{equation}
 where $\widetilde{C}(T)=\|u_{0}\|_{L^{2}}^{2}+4T\|f\|_{\dot{H}^{-\frac{\alpha}{2}}}^{2}$.

 To finish the proof, it's necessary to deduce the following strong convergence
  \begin{equation}\label{e3.4}
u^{\varepsilon}_{n}\rightarrow u^{\varepsilon} ~~in ~~L^{2}([0,T),L^{2}(-l,l))
 \end{equation}
 for any constant $l>0$, which will be used to deal with the convergence of nonlinear term. Here we use the Lemma \ref{lem2.4} by choosing $X=H^{1}(-l,l)$, $Y=L^{2}(-l,l)$ and $Z=H^{-1}(-l,l)$ and let $p=q=2$. Based on the estimate (\ref{e3.3}), it's sufficient to estimate $\|\partial_{t}u^{\varepsilon}_{n}\|_{L^{2}(0,T;H^{-1}(-l,l))}$. Multiplying (\ref{e3.1}) by $\phi\in H^{1}$ and integrating on the real line give
\begin{align}\label{e3.5}
 |\langle \partial_{t}u^{\varepsilon}_{n},\phi \rangle|
 &\leq|\langle u^{\varepsilon}_{n}\partial_{x} u^{\varepsilon}_{n},\phi \rangle|+|\langle \Lambda^{\frac{\alpha}{2}}u^{\varepsilon}_{n},\Lambda^{\frac{\alpha}{2}}\phi\rangle|+\varepsilon|\langle \partial_{x}u^{\varepsilon}_{n},\phi_{x} \rangle|+|\langle f,\phi\rangle| \nonumber\\
 &\leq \|u^{\varepsilon}_{n}\|_{L^{2}}\|\partial_{x}u^{\varepsilon}_{n}\|_{L^{2}}\|\phi\|_{L^{\infty}}+
 \|\Lambda^{\frac{\alpha}{2}}u^{\varepsilon}_{n}\|_{L^{2}}\|\Lambda^{\frac{\alpha}{2}}\phi\|_{L^{2}}+
 \varepsilon\|\partial_{x}u^{\varepsilon}_{n}\|_{L^{2}}\|\phi_{x}\|_{L^{2}} \nonumber\\
 &+\|f\|_{L^{2}}\|\phi\|_{L^{2}}  \nonumber\\
 &\leq \|\phi\|_{H^{1}}(\|u^{\varepsilon}_{n}\|_{L^{2}}\|u^{\varepsilon}_{n}\|_{H^{1}}+
 \|u^{\varepsilon}_{n}\|_{H^{\frac{\alpha}{2}}}+
 \varepsilon\|u^{\varepsilon}_{n}\|_{H^{1}}+\|f\|_{L^{2}}).
 \end{align}
From (\ref{e3.3}) and (\ref{e3.5}), we have
 \begin{align}\label{e3.6}
\int^{T}_{0}\|\partial_{t}u^{\varepsilon}_{n}\|_{H^{-1}}^{2}dt
 &\leq c\int^{T}_{0} (\|u^{\varepsilon}_{n}\|_{L^{2}}^{2}\|u^{\varepsilon}_{n}\|_{H^{1}}^{2}+
 \|u^{\varepsilon}_{n}\|_{H^{\frac{\alpha}{2}}}^{2}+\varepsilon\|u^{\varepsilon}_{n}\|_{H^{1}}^{2}+\|f\|_{L^{2}}^{2})dt \nonumber\\
 &\leq c (\sup_{t\geq0}\|u^{\varepsilon}_{n}\|_{L^{2}}^{2}+\varepsilon)\int^{T}_{0}\|u^{\varepsilon}_{n}\|_{H^{1}}^{2}dt+
 c\int^{T}_{0}\|u^{\varepsilon}_{n}\|_{H^{\frac{\alpha}{2}}}^{2}dt+c \|f\|_{L^{2}}^{2}T \nonumber\\
 &\leq C(c,\widetilde{C}(T)).
 \end{align}
Thus the strong convergence of $\{u^{\varepsilon}_{n}\}$ in (\ref{e3.4}) follows from Lemma \ref{lem2.4}.
\end{proof}

It's easy to prove the existence of the viscosity weak solution for (\ref{e11}) by taking the limit in Lemma \ref{lem3.3}, which can be stated as follows.
\begin{cor}\label{cor3.4}
If $u_{0}\in L^{2}(R)$ and $f\in \dot{H}^{-\frac{\alpha}{2}}(R)\cap L^{2}(R)$, then there exists a global viscosity weak solution $u\in L^{\infty}([0,T),L^{2}(R))\cap L^{2}([0,T),\dot{H}^{\frac{\alpha}{2}}(R))$ to (\ref{e11}), which satisfies for any given $T>0$,
\begin{equation}
\begin{aligned}
~&-\int^{T}_{0}\langle u, \psi_{t}\rangle dt-\frac{1}{2}\int^{T}_{0}\langle u^{2},\psi_{x}\rangle dt+\int^{T}_{0}\langle\Lambda^{\frac{\alpha}{2}}u,
\Lambda^{\frac{\alpha}{2}}\psi\rangle dt \nonumber\\
&=\langle u_{0},\psi(x,0)\rangle+\int^{T}_{0}\langle f,\psi\rangle dt, \quad for  ~~ \psi\in C_{c}^{\infty}([0,T)\times R).
 \end{aligned}
\end{equation}
Moreover, the estimate
$$
\sup_{t\geq 0}\|u(t)\|_{L^{2}}^{2}+\int^{T}_{0}\|\Lambda^{\frac{\alpha}{2}} u(t)\|_{L^{2}}^{2}dt\leq \|u_{0}\|_{L^{2}}^{2}+4T\|\Lambda^{-\frac{\alpha}{2}}f\|_{L^{2}}^{2}
$$
holds.
\end{cor}

\section{Uniqueness of steady-state solution with finite $H^{\frac{\alpha}{2}}$ energy}
The main aim in this section is to establish the existence and uniqueness of weak solution to steady-state equation (\ref{e12}) with $\alpha\in(1,\frac{3}{2+\epsilon})$ for any given small positive number $\epsilon$. Without loss of generality, we normalize the Sobolev embedding constant in the following.

\subsection{The outline of the analysis scheme}
 Motivated by the work \cite{6} and \cite{7}, we find that if $u$ solves
\begin{equation}\label{e4.1}
\left\{ \begin{array}{ll}
u_{t}+\frac{1}{2}(Vu)_{x}+\Lambda^{\alpha}u=0,  \\
u(0,x)=f(x),
 \end{array} \right.
\end{equation}
then $U=\int^{\infty}_{0}u(t)dt$ solves
\begin{equation}\label{e4.2}
\frac{1}{2}(VU)_{x}+\Lambda^{\alpha}U=f.
\end{equation}
Since this equation is linear for a given $V$, the solution is unique and hence $U=V$. Furthermore, it follows from the integral
Minkowski's inequality that
$$
\|U\|_{L^{2}}=\|\int^{\infty}_{0}u(t)dt\|_{L^{2}}\leq\int^{\infty}_{0}\|u(t)\|_{L^{2}}dt
$$
and $U\in L^{2}$ if $\|u(t)\|_{L^{2}}\leq C(1+t)^{-\beta}$ with $\beta>1$. Thus it's sufficient to establish a fast decay for $\|u(t)\|_{L^{2}}$ to obtain bounded $H^{\frac{\alpha}{2}}$ energy of $U$.

\subsection{The proof of Theorem \ref{thm1.4}:}
\begin{proof}
The proof will be divided into the following four steps.
Now let's consider the following two sequences of the approximating equations
\begin{equation}\label{e4.3}
\frac{1}{2}(U^{i}U^{i+1})_{x}+\Lambda^{\alpha}U^{i+1}=f.
\end{equation}
and
\begin{equation}\label{e4.4}
\left\{ \begin{array}{ll}
\partial_{t}u^{i+1}+\frac{1}{2}(U^{i}u^{i+1})_{x}+\Lambda^{\alpha}u^{i+1}=0,  \\
u^{i+1}(0,x)=f(x).
 \end{array} \right.
\end{equation}
Fixing $U^{i}\in H^{\frac{\alpha}{2}}$ with $\alpha\in(1,\frac{3}{2+\epsilon})$, we would solve these two equations recursively to obtain the approximating solutions in the first step. The finite $L^{2}$ norm of the solution $U^{i+1}$ to (\ref{e4.3}) mainly relies on the decay estimate for $u^{i+1}$, which will be shown in the second step. The last two steps are devoted to deal with the $H^{\frac{\alpha}{2}}$ boundedness of $U^{i+1}$ and its convergence.\\
~\\
{\bf Step 1~~ Existence of solutions to the approximating equations}
\begin{lem}\label{lem4.1}
For any given $0<\epsilon<1$, assume $U^{i}\in H^{\frac{\alpha}{2}}$ for $\alpha\in(1,\frac{3}{2+\epsilon})$ with
$$
\|U^{i}\|_{H^{\frac{\alpha}{2}}}\leq C\epsilon^{-1}\|f\|_{X},
$$
where $C:=max\{3\epsilon, \frac{4\sqrt{\alpha \epsilon(12-2\alpha \epsilon)}}{3-2\alpha-\alpha \epsilon}\}$, then there exists a unique weak solution $U^{i+1}$ to (\ref{e4.3}) in the sense that for $\psi\in C^{\infty}_{0}(R)$,
$$
-\frac{1}{2}\langle U^{i}U^{i+1}, \psi_{x}\rangle+\langle \Lambda^{\frac{\alpha}{2}}U^{i+1}, \Lambda^{\frac{\alpha}{2}}\psi\rangle=\langle f, \psi\rangle.
$$
Moreover, it satisfies
\begin{equation}\label{e4.5}
\|\Lambda^{\frac{\alpha}{2}}U^{i+1}\|_{L^{2}}\leq \frac{C}{2}\varepsilon^{-1}\|f\|_{X}.
\end{equation}
\end{lem}
\begin{proof}
Since the equation (\ref{e4.3}) is linear for a given $U^{i}$, a standard Galerkin method gives the existence and uniqueness of solutions $U^{i+1}$. Multiplying (\ref{e4.3}) by $U^{i+1}$ and integrating on $R$ yield
\begin{equation}
\begin{aligned}
\|\Lambda^{\frac{\alpha}{2}}U^{i+1}\|_{L^{2}}^{2}
&\leq\frac{1}{2}|\int_{R}U^{i}U^{i+1}\partial_{x}U^{i+1}dx|+|\int_{R}fU^{i+1}dx| \nonumber\\
&\leq\frac{1}{2}|\int_{R}\Lambda^{\frac{2-\alpha}{2}}(U^{i}U^{i+1})\Lambda^{\frac{\alpha}{2}}U^{i+1}dx|+\|\Lambda^{-\frac{\alpha}{2}}f\|_{L^{2}}
\|\Lambda^{\frac{\alpha}{2}}U^{i+1}\|_{L^{2}} \nonumber\\
&\leq\frac{1}{2}\|\Lambda^{1-\frac{\alpha}{2}}(U^{i}U^{i+1})\|_{L^{2}}\|\Lambda^{\frac{\alpha}{2}}U^{i+1}\|_{L^{2}}+
\|\Lambda^{-\frac{\alpha}{2}}f\|_{L^{2}}
\|\Lambda^{\frac{\alpha}{2}}U^{i+1}\|_{L^{2}}
 \end{aligned}
\end{equation}
By using the Lemma \ref{lem2.1}--Lemma \ref{lem2.3}, we have
\begin{equation}
\begin{aligned}
\|\Lambda^{1-\frac{\alpha}{2}}(U^{i}U^{i+1})\|_{L^{2}}
&\leq \|\Lambda^{1-\frac{\alpha}{2}}U^{i}\|_{L^{2}}\|U^{i+1}\|_{L^{\infty}}+\|U^{i}\|_{L^{\frac{1}{\alpha-1}}}
\|\Lambda^{1-\frac{\alpha}{2}}U^{i+1}\|_{L^{\frac{2}{3-2\alpha}}}\nonumber\\
&\leq C(\|\Lambda^{1-\frac{\alpha}{2}}U^{i}\|_{L^{2}}\|\Lambda^{\frac{\alpha}{2}}U^{i+1}\|_{L^{2}}+
\|U^{i}\|_{L^{\frac{1}{\alpha-1}}}\|\Lambda^{\alpha-1}(\Lambda^{1-\frac{\alpha}{2}}U^{i+1})\|_{L^{2}})
 \nonumber \\
&= C(\|\Lambda^{1-\frac{\alpha}{2}}U^{i}\|_{L^{2}}+
\|U^{i}\|_{L^{\frac{1}{\alpha-1}}})\|\Lambda^{\frac{\alpha}{2}}U^{i+1}\|_{L^{2}}
 \nonumber \\
&\leq C\|U^{i}\|_{H^{\frac{\alpha}{2}}}\|\Lambda^{\frac{\alpha}{2}}U^{i+1}\|_{L^{2}}
\leq C\epsilon^{-1}\|f\|_{X}\|\Lambda^{\frac{\alpha}{2}}U^{i+1}\|_{L^{2}}.
 \end{aligned}
\end{equation}
Thus, choosing $C(\alpha,\epsilon)$ and $\|f\|_{X}\leq C(\alpha,\epsilon)$
such that
 $C\epsilon^{-1}\|f\|_{X}\leq \frac{1}{3}$, then (\ref{e4.5}) follows from
\begin{equation}\label{e4.6}
\|\Lambda^{\frac{\alpha}{2}}U^{i+1}\|_{L^{2}}\leq \frac{3}{2}\|f\|_{X}\leq \frac{C}{2}\epsilon^{-1}\|f\|_{X}.
\end{equation}
\end{proof}

\begin{lem}\label{lem4.2}
Assume $U^{i}\in H^{\frac{\alpha}{2}}$ for $\alpha\in(1,\frac{3}{2+\epsilon})$ with
$$
\|U^{i}\|_{H^{\frac{\alpha}{2}}}\leq C\epsilon^{-1}\|f\|_{X},
$$
where $C:=max\{3\epsilon, \frac{4\sqrt{\alpha \epsilon(12-2\alpha \epsilon)}}{3-2\alpha-\alpha \epsilon}\}$, then there exists a unique weak solution $u^{i+1}$ to (\ref{e4.4}) in the sense that for $\psi\in C^{\infty}_{0}(R^{+}\times R)$,
$$
-\int^{\infty}_{0}\langle u^{i+1}, \psi_{t}\rangle dt-\frac{1}{2}\int^{\infty}_{0}\langle U^{i}u^{i+1},\psi_{x}\rangle dt+\int^{\infty}_{0}\langle
\Lambda^{\frac{\alpha}{2}}u^{i+1},\Lambda^{\frac{\alpha}{2}}\psi\rangle dt=\langle f,\psi\rangle.
$$
Moreover, it satisfies
\begin{equation}\label{e4.7}
\sup_{t\geq 0}\|u^{i+1}\|_{L^{2}}^{2}+\frac{4}{3}\int^{t}_{0}\|\Lambda^{\frac{\alpha}{2}}u^{i+1}\|_{L^{2}}^{2}ds\leq\|f\|_{X}^{2}.
\end{equation}
\end{lem}
\begin{proof}
A sequence of approximating solutions $u^{i+1}$ can be got by using the Galerkin method, (The rigorous proof should be carried out by using a sequence of the Galerkin smooth approximating solutions $u^{i+1}_{n}$, for convenience we work on $u^{i+1}$ again.) which satisfies
$$
\frac{1}{2}\frac{d}{dt}\int_{R}|u^{i+1}|^{2}dx+\int_{R}|\Lambda^{\frac{\alpha}{2}}u^{i+1}|^{2}dx=-\frac{1}{2}\int_{R}(U^{i}u^{i+1})_{x}u^{i+1}dx.
$$
Similarly, by applying the Lemma \ref{lem2.1}--Lemma \ref{lem2.3}, we have
\begin{equation}
\begin{aligned}
|\int_{R}(U^{i}u^{i+1})_{x}u^{i+1}dx|
&\leq|\int_{R}\Lambda^{1-\frac{\alpha}{2}}(U^{i}u^{i+1})\Lambda^{\frac{\alpha}{2}}u^{i+1}dx|  \nonumber\\
&\leq\|\Lambda^{1-\frac{\alpha}{2}}(U^{i}u^{i+1})\|_{L^{2}}\|\Lambda^{\frac{\alpha}{2}}u^{i+1}\|_{L^{2}}\nonumber\\
&\leq C\|U^{i}\|_{H^{\frac{\alpha}{2}}}\|\Lambda^{\frac{\alpha}{2}}u^{i+1}\|_{L^{2}}^{2}\leq C\epsilon^{-1}\|f\|_{X}\|\Lambda^{\frac{\alpha}{2}}u^{i+1}\|_{L^{2}}^{2}.
 \end{aligned}
\end{equation}
We choose $C(\alpha,\epsilon)$ and $\|f\|_{X}\leq C(\alpha,\epsilon)$,
such that
 $C\epsilon^{-1}\|f\|_{X}\leq \frac{1}{3}$.
Hence
\begin{equation}\label{e4.8}
\frac{1}{2}\frac{d}{dt}\int_{R}|u^{i+1}|^{2}dx+\frac{2}{3}\int_{R}|\Lambda^{\frac{\alpha}{2}}u^{i+1}|^{2}dx\leq0.
\end{equation}
Integrating on both sides of (\ref{e4.8}) over $[0,t)$  gives (\ref{e4.7}).
\end{proof}
~\\
{\bf Step 2: The decay of $u^{i+1}$}
\begin{lem}\label{lem4.3}
Let $u^{i+1}$ be the solution of (\ref{e4.4}) obtained in Lemma \ref{lem4.2}, assume $\|U^{i}\|_{H^{\frac{\alpha}{2}}}\leq C\varepsilon^{-1}\|f\|_{X}$ where $C:=max\{3\epsilon, \frac{4\sqrt{\alpha \epsilon(12-2\alpha \epsilon)}}{3-2\alpha-\alpha \epsilon}\}$,
and the assumptions in Theorem \ref{thm1.4} hold, then we have
$$
\|u^{i+1}(t)\|_{L^{2}}\leq \frac{12-2\alpha\epsilon}{\alpha\epsilon}\|f\|_{X}(1+t)^{-(\frac{3}{2\alpha}-\frac{\epsilon}{2})}.
$$
\end{lem}
\begin{rem}
It's worth noting that $\frac{3}{2\alpha}-\frac{\epsilon}{2}>1$ due to the assumption $\alpha\in (1,\frac{3}{2+\varepsilon})$.
\end{rem}
\begin{proof}
Here we mainly use some new observations and the Fourier splitting method, which was used first by Schonbek in \cite{8} on the decay of solutions for parabolic equations. Later on the method was used widely to obtain several results for the Navier-Stokes equations and so on, cf \cite{9}. Multiplying (\ref{e4.4}) by $u^{i+1}$  and integrating on $R$ yields
 \begin{equation}
\begin{aligned}
\frac{1}{2}\frac{d}{dt}\int_{R}|u^{i+1}|^{2}dx+\int_{R}|\Lambda^{\frac{\alpha}{2}}u^{i+1}|^{2}dx
&=-\frac{1}{2}\int_{R}(U^{i}u^{i+1})_{x}u^{i+1}dx  \nonumber\\
&\leq\|U^{i}\|_{H^{\frac{\alpha}{2}}}\|\Lambda^{\frac{\alpha}{2}}u^{i+1}\|_{L^{2}}^{2}\nonumber\\
&\leq C\epsilon^{-1}\|f\|_{X}\|\Lambda^{\frac{\alpha}{2}}u^{i+1}\|_{L^{2}}^{2}.
 \end{aligned}
  \end{equation}
  Choosing $C(\alpha, \epsilon)$, such that $C\epsilon^{-1}\|f\|_{X}\leq \frac{1}{3}$,
 that is to say,
 \begin{equation}\label{e4.9}
 \frac{d}{dt}\int_{R}|u^{i+1}|^{2}dx+\frac{4}{3}\|\Lambda^{\frac{\alpha}{2}}u^{i+1}\|_{L^{2}}^{2}\leq 0.
 \end{equation}
Denote $S(t)=\{\xi|~ |\xi|\leq g(t)\}$ with $g(t)=(\frac{3m}{4(1+t)})^{\frac{1}{\alpha}}$, where $m$ is a undetermined constant. Here we decompose the dissipative term in frequency space into
two parts $S(t)$ and $S(t)^{c}$. It follows from the Plancherel theorem that one has
\begin{equation}
\begin{aligned}
\|\Lambda^{\frac{\alpha}{2}}u^{i+1}\|_{L^{2}}^{2}
&=\int_{R}|\xi|^{\alpha}|\widehat{u^{i+1}}|^{2}d\xi=\int_{S(t)}|\xi|^{\alpha}|\widehat{u^{i+1}}|^{2}d\xi+\int_{S(t)^{c}}|\xi|^{\alpha}|\widehat{u^{i+1}}|^{2}d\xi  \nonumber\\
&\geq\int_{S(t)^{c}}|\xi|^{\alpha}|\widehat{u^{i+1}}|^{2}d\xi\geq\int_{S(t)^{c}}|g(t)|^{\alpha}|\widehat{u^{i+1}}|^{2}d\xi \nonumber\\
&=\int_{R}|g(t)|^{\alpha}|\widehat{u^{i+1}}|^{2}d\xi-\int_{S(t)}|g(t)|^{\alpha}|\widehat{u^{i+1}}|^{2}d\xi
 \end{aligned}
\end{equation}
Combining this inequality with (\ref{e4.9}), we obtain
\begin{equation}\label{e4.10}
\frac{d}{dt}\int_{R}|u^{i+1}|^{2}dx+\frac{4}{3}|g(t)|^{\alpha}\int_{R}|\widehat{u^{i+1}}|^{2}d\xi
\leq\frac{4}{3}|g(t)|^{\alpha}\int_{S(t)}|\widehat{u^{i+1}}|^{2}d\xi
 \end{equation}
For the given $U^{i}$, it's easy to express the implicit solutions to (\ref{e4.4}) by
 $$
 u^{i+1}(x,t)=e^{-\Lambda^{\alpha}t}f(x)-\frac{1}{2}e^{-\Lambda^{\alpha}t}\int^{t}_{0}e^{\Lambda^{\alpha}s}(U^{i}u^{i+1})_{x}ds.
 $$
 Taking the Fourier transform on both sides of above formula yields
 $$
 \widehat{u^{i+1}}(\xi)=e^{-|\xi|^{\alpha}t}\hat{f}(\xi)-\frac{1}{2}\int^{t}_{0}e^{-|\xi|^{\alpha}(t-s)}i\xi\widehat{U^{i}u^{i+1}}(\xi)ds
 $$
 Considering $supp\hat{f}$ far from the origin, we have
\begin{equation}
\begin{aligned}
 |\widehat{u^{i+1}}(\xi)|
&\leq |e^{-|\xi|^{\alpha}t}\hat{f}(\xi)|+\frac{1}{2}\int^{t}_{0}|\xi||\widehat{U^{i}u^{i+1}}|(\xi)ds \nonumber\\
&\leq e^{-\rho^{\alpha} t}|\hat{f}(\xi)|+\frac{1}{2}|\xi|\int^{t}_{0}|\widehat{U^{i}u^{i+1}}|(\xi)ds.
\end{aligned}
\end{equation}
Hence
\begin{align}\label{e4.11}
 \int_{S(t)}|\widehat{u^{i+1}}(\xi)|^{2}d\xi
 &\leq 2e^{-2\rho^{\alpha} t}\int_{S_{(t)}}|\hat{f}(\xi)|^{2}d\xi+\frac{1}{2}\int_{S_{(t)}}|\xi|^{2}d\xi(\int^{t}_{0}\|U^{i}\|_{L^{2}}\|u^{i+1}\|_{L^{2}}ds)^{2} \nonumber\\
 &\leq 2e^{-2\rho^{\alpha} t}\|f\|_{X}^{2}+\frac{1}{3}g^{3}(t)\|U^{i}\|_{L^{2}}^{2}(\int^{t}_{0}\|u^{i+1}(s)\|_{L^{2}}ds)^{2}.
\end{align}
(\ref{e4.10}) and (\ref{e4.11}) show that
\begin{align}\label{e4.12}
~&
\frac{d}{dt}\int_{R}|u^{i+1}|^{2}dx+\frac{4}{3} g^{\alpha}(t)\int_{R}|\widehat{u^{i+1}}|^{2}d\xi \nonumber\\
 &\leq \frac{8}{3} g^{\alpha}(t)e^{-2\rho^{\alpha} t}\|f\|_{X}^{2}+\frac{4}{9} g^{3+\alpha}(t)\|U^{i}\|_{L^{2}}^{2}(\int^{t}_{0}\|u^{i+1}(s)\|_{L^{2}}ds)^{2}.
\end{align}
Multiplying $(1+t)^{m}$ on both sides of (\ref{e4.12}), one can deduce
\begin{align}\label{e4.13}
~&
\frac{d}{dt}[(1+t)^{m}\|u^{i+1}\|_{L^{2}}^{2}]\leq 2m(1+t)^{m-1}e^{-2\rho^{\alpha} t}\|f\|_{X}^{2} \nonumber\\
 &+\frac{4}{9}(\frac{3m}{4})^{\frac{3+\alpha}{\alpha}}(1+t)^{m-\frac{3+\alpha}{\alpha}}
 \|U^{i}\|_{L^{2}}^{2}(\int^{t}_{0}\|u^{i+1}(s)\|_{L^{2}}ds)^{2}.
\end{align}
Letting $m=\frac{3}{\alpha}-\epsilon$ for small $\epsilon>0$ in (\ref{e4.13}) yields
\begin{align}\label{e4.14}
~&
\frac{d}{dt}[(1+t)^{\frac{3-\alpha\epsilon}{\alpha}}\|u^{i+1}\|_{L^{2}}^{2}]\leq \frac{6-2\alpha\epsilon}{\alpha}(1+t)^{-1-\epsilon}(1+t)^{\frac{3}{\alpha}}e^{-2\rho^{\alpha} t}\|f\|_{X}^{2}+ \nonumber\\
 &\frac{4}{9}(\frac{9-3\alpha\epsilon}{4\alpha})^{\frac{3+\alpha}{\alpha}}(1+t)^{-1-\epsilon}
 \|U^{i}\|_{L^{2}}^{2}(\int^{t}_{0}\|u^{i+1}(s)\|_{L^{2}}ds)^{2} \nonumber\\
 &\leq \frac{6-2\alpha\epsilon}{\alpha}(1+t)^{-1-\epsilon}\|f\|_{X}^{2}+\frac{4}{9}(\frac{9-3\alpha\epsilon}{4\alpha})^{\frac{3+\alpha}{\alpha}}(1+t)^{-1-\epsilon}
 \|U^{i}\|_{L^{2}}^{2}(\int^{t}_{0}\|u^{i+1}(s)\|_{L^{2}}ds)^{2}\nonumber\\
  &\leq \frac{6-2\alpha\epsilon}{\alpha}(1+t)^{-1-\epsilon}\|f\|_{X}^{2}+\nonumber\\
 &\frac{4}{9}(\frac{9-3\alpha\epsilon}{4\alpha})^{\frac{3+\alpha}{\alpha}}(1+t)^{-1-\epsilon}
 \|U^{i}\|_{L^{2}}^{2}(\int^{t}_{0}(1+s)^{\frac{3-\alpha\epsilon}{2\alpha}}(1+s)^{-\frac{3-\alpha\epsilon}{2\alpha}}\|u^{i+1}(s)\|_{L^{2}}ds)^{2}\nonumber\\
 &\leq \frac{6-2\alpha\epsilon}{\alpha}(1+t)^{-1-\epsilon}\|f\|_{X}^{2}+\nonumber\\
  &\frac{4}{9}(\frac{9-3\alpha\epsilon}{4\alpha})^{\frac{3+\alpha}{\alpha}}(1+t)^{-1-\epsilon}
 \|U^{i}\|_{L^{2}}^{2}\sup_{0\leq s\leq t}[(1+s)^{\frac{3-\alpha\epsilon}{\alpha}}\|u^{i+1}(s)\|_{L^{2}}^{2}](\int^{t}_{0}(1+s)^{-\frac{3-\alpha\epsilon}{2\alpha}}ds)^{2}.
\end{align}
From $1<\alpha<\frac{3}{2+\epsilon}$, we have $\frac{3-\alpha\epsilon}{2\alpha}=\frac{3}{2\alpha}-\frac{\varepsilon}{2}>1$, which implies
\begin{equation}\label{e4.15}
(\int^{t}_{0}(1+s)^{-\frac{3-\alpha\epsilon}{2\alpha}}ds)^{2}\leq (\frac{2\alpha}{3-(2+\epsilon)\alpha})^{2}.
\end{equation}
Integrating on both sides of (\ref{e4.14}) over $(0,t)$, it follows from (\ref{e4.15}) that
\begin{align}\label{e4.16}
~&
(1+t)^{\frac{3-\alpha\epsilon}{\alpha}}\|u^{i+1}\|_{L^{2}}^{2}-\|f\|_{L^{2}}^{2} \nonumber\\
&\leq \frac{6-2\alpha\epsilon}{\alpha}\|f\|_{X}^{2}\int^{t}_{0}(1+s)^{-1-\epsilon}ds+ \nonumber\\
&\frac{4}{9}(\frac{9-3\alpha\epsilon}{4\alpha})^{\frac{3+\alpha}{\alpha}}(\frac{2\alpha}{3-(2+\epsilon)\alpha})^{2}
 \|U^{i}\|_{L^{2}}^{2}\sup_{0\leq s\leq t}[(1+s)^{\frac{3-\alpha\epsilon}{\alpha}}\|u^{i+1}(s)\|_{L^{2}}^{2}]\int^{t}_{0}(1+s)^{-1-\epsilon}ds \nonumber\\
 &\leq \frac{6-2\alpha\epsilon}{\alpha\epsilon}\|f\|_{X}^{2} + \nonumber\\
&\frac{4}{9\epsilon}(\frac{9-3\alpha\epsilon}{4\alpha})^{\frac{3+\alpha}{\alpha}}(\frac{2\alpha}{3-(2+\epsilon)\alpha})^{2}
 \|U^{i}\|_{L^{2}}^{2}\sup_{0\leq s\leq t}[(1+s)^{\frac{3-\alpha\epsilon}{\alpha}}\|u^{i+1}(s)\|_{L^{2}}^{2}].
\end{align}
Due to $\|U^{i}\|_{L^{2}}\leq \|U^{i}\|_{H^{\frac{\alpha}{2}}}\leq C\epsilon^{-1}\|f\|_{X}$ and $\|f\|_{X}\leq C(\alpha,\epsilon)$, then we can choose suitable $C(\alpha,\epsilon)$  such that $\frac{4}{9\epsilon}(\frac{9-3\alpha\epsilon}{4\alpha})^{\frac{3+\alpha}{\alpha}}(\frac{2\alpha}{3-(2+\epsilon)\alpha})^{2}
 \|U^{i}\|_{L^{2}}^{2}\leq \frac{1}{2}$. Hence (\ref{e4.16}) gives
\begin{equation}\label{e4.17}
(1+t)^{\frac{3-\alpha\epsilon}{\alpha}}\|u^{i+1}\|_{L^{2}}^{2}
 \leq (\frac{6-2\alpha\epsilon}{\alpha\epsilon}+1)\|f\|_{X}^{2}+\frac{1}{2}Q(t),\quad for~~~t\geq0,
\end{equation}
where $Q(t)=\sup_{0\leq s\leq t}[(1+s)^{\frac{3-\alpha\epsilon}{\alpha}}\|u^{i+1}(s)\|_{L^{2}}^{2}]$. Now taking the supremum for $t\in [0,T]$ on both sides of (\ref{e4.17}), we obtain
$$
Q(T)\leq (\frac{6-2\alpha\epsilon}{\alpha\epsilon}+1)\|f\|_{X}^{2}+\frac{1}{2}Q(T),\quad for~~~T\geq0.
$$
Therefore, combining (\ref{e4.17}), we obtain
$$
(1+t)^{\frac{3-\alpha\epsilon}{\alpha}}\|u^{i+1}\|_{L^{2}}^{2}
 \leq (\frac{6-2\alpha\epsilon}{\alpha\epsilon}+1)\|f\|_{X}^{2}+(\frac{6-2\alpha\epsilon}{\alpha\epsilon}+1)\|f\|_{X}^{2}<\infty,
$$
which yields
\begin{equation}\label{e4.18}
\|u^{i+1}(t)\|_{L^{2}}^{2}\leq \frac{12-2\alpha\epsilon}{\alpha\epsilon}\|f\|_{X}^{2}(1+t)^{-\frac{3}{\alpha}+\epsilon}.
\end{equation}
Combining $1<\alpha<\frac{3}{2+\epsilon}$ with (\ref{e4.18}), we have
$$
\|u^{i+1}(t)\|_{L^{2}}\leq \sqrt{\frac{12-2\alpha\epsilon}{\alpha\epsilon}}\|f\|_{X}(1+t)^{-(\frac{3}{2\alpha}-\frac{\epsilon}{2})}, ~~where~~\frac{3}{2\alpha}-\frac{\epsilon}{2}>1.
$$
\end{proof}
~\\
{\bf Step 3~~ Proving $U^{i+1}=\int^{\infty}_{0}u^{i+1}(t)dt$ solves (\ref{e4.3}) }

We have noted  that $U^{i+1}=\int^{\infty}_{0}u^{i+1}(t)dt$ solves formally the approximating equation (\ref{e4.3}). Now we prove
this rigorously and show that $U^{i+1}$ is uniformally bounded in $H^{\frac{\alpha}{2}}$.
\begin{lem}\label{lem4.4}
Let $u^{i+1}$ be the solution of (\ref{e4.4}) obtained in Lemma \ref{lem4.2},
then $\int^{\infty}_{0}u^{i+1}(t)dt\in L^{2}(R)$ solves uniquely (\ref{e4.3}).
Moreover,
$$
\int^{\infty}_{0}u^{i+1}(t)dt=U^{i+1}.
$$
\end{lem}
\begin{proof}
For each $i$, define the sequence $\{V^{i+1}_{n}\}$ by
$$
V^{i+1}_{n}=\int^{n}_{0}u^{i+1}(t)dt.
$$
The Minkowski's inequality and Lemma \ref{lem4.3} yield
$$
\|V^{i+1}_{n}\|_{L^{2}}\leq\int^{n}_{0}\|u^{i+1}(t)\|_{L^{2}}dt\leq\frac{24-4\epsilon\alpha}{(3-2\alpha-\epsilon\alpha)\epsilon}\|f\|_{X}.
$$
Similarly, by applying the Minkowski's inequality again, we have
$$
\|V^{i+1}_{n+1}-V^{i+1}_{n}\|_{L^{2}}\leq \int_{n}^{n+1}\|u^{i+1}(t)\|_{L^{2}}dt.
$$
It follows from Lemma \ref{lem4.3} that
$$
\int_{n}^{n+1}\|u^{i+1}(t)\|_{L^{2}}dt\leq \frac{12-2\alpha\epsilon}{\alpha\epsilon}\|f\|_{X}\int_{n}^{n+1}(1+t)^{-\frac{3}{2\alpha}+\frac{\epsilon}{2}}dt \rightarrow 0,  ~~  as~~n\rightarrow \infty.
$$
Thus $\{V^{i+1}_{n}\}$ is a Cauchy consequence in $L^{2}$ and
$$
V^{i+1}_{n}=\int^{n}_{0}u^{i+1}(t)dt\rightarrow V^{i+1} \quad in~~L^{2},  ~~as~~n\rightarrow \infty.
$$
Denote $V^{i+1}=\int^{\infty}_{0}u^{i+1}(t)dt$, it's easy to check that $V^{i+1}$ is a weak solution of (\ref{e4.3}). The uniqueness in Lemma \ref{lem4.1} shows that $V^{i+1}=U^{i+1}$.
\end{proof}

\begin{lem}\label{lem4.5}
Let $U^{i+1}$ be the solution of (\ref{e4.3}) obtained in Lemma \ref{lem4.1} and the assumptions in Theorem \ref{thm1.4} hold,
then
$$
\|U^{i+1}\|_{H^{\frac{\alpha}{2}}}\leq C\epsilon^{-1}\|f\|_{X}.
$$
\end{lem}
\begin{proof}
From Lemma \ref{lem4.3}, Lemma \ref{lem4.4} and (\ref{e4.5}), we can deduce
\begin{equation}
\begin{aligned}
 \|U^{i+1}\|_{H^{\frac{\alpha}{2}}}
 &=\|U^{i+1}\|_{L^{2}}+\|\Lambda^{\frac{\alpha}{2}}U^{i+1}\|_{L^{2}} \nonumber\\
 &\leq \int^{\infty}_{0}\|u^{i+1}(t)\|_{L^{2}}dt+\frac{C}{2}\epsilon^{-1}\|f\|_{X} \nonumber\\
 &\leq \sqrt{\frac{12-2\alpha\epsilon}{\alpha\epsilon}}\|f\|_{X}\int^{\infty}_{0}(1+t)^{-\frac{3}{2\alpha}+\frac{\epsilon}{2}}dt+\frac{C}{2}\varepsilon^{-1}\|f\|_{X}\nonumber\\
 &\leq (\sqrt{\frac{12-2\alpha\epsilon}{\alpha\epsilon}}\frac{2\alpha}{3-2\alpha-\epsilon\alpha}+\frac{C}{2}\epsilon^{-1})\|f\|_{X}\leq C\varepsilon^{-1}\|f\|_{X}.
\end{aligned}
\end{equation}
\end{proof}
~\\
{\bf Step 4~~ Proving $U^{i+1}\rightarrow U$ is the unique weak solution of (\ref{e12})}\\
~\\
In this step, we would finish the proof of Theorem \ref{thm1.4} by showing that the approximating solutions $U^{i+1}$ to (\ref{e4.3}) converge to a limit function $U$, which solves uniquely (\ref{e12}) with $\|U\|_{H^{\frac{\alpha}{2}}}\leq C\epsilon^{-1}\|f\|_{X}$. Starting with $U^{0}$, we can solve (\ref{e4.3}) recursively by Lemma \ref{lem4.1} and obtain a sequence $\{U^{i}\}_{i=0}^{\infty}$, which satisfies $\|U^{i}\|_{H^{\frac{\alpha}{2}}}\leq C\epsilon^{-1}\|f\|_{X}$ due to Lemma \ref{lem4.5}. Now we first show the sequence $\{U^{i}\}$ is Cauchy sequence in $\dot{H}^{\frac{\alpha}{2}}$. It's obvious that there holds
\begin{equation}\label{e4.19}
\frac{1}{2}\partial_{x}(U^{i}U^{i+1})+\Lambda^{\alpha}U^{i+1}=f
\end{equation}
and
\begin{equation}\label{e4.20}
\frac{1}{2}\partial_{x}(U^{i-1}U^{i})+\Lambda^{\alpha}U^{i}=f.
\end{equation}
Let $Y^{i+1}=U^{i+1}-U^{i}$, then (\ref{e4.19}) and (\ref{e4.20}) imply that $Y^{i+1}$ satisfies
\begin{equation}\label{e4.21}
\frac{1}{2}\partial_{x}(U^{i}Y^{i+1})+\frac{1}{2}\partial_{x}(U^{i}Y^{i})+\Lambda^{\alpha}Y^{i+1}=0.
\end{equation}
Since $Y^{i+1}\in H^{\frac{\alpha}{2}}(R)$, then multiplying the equation (\ref{e4.21}) by $Y^{i+1}$ and integrating on $R$ yield
$$
\frac{1}{2}\int_{R}\partial_{x}(U^{i}Y^{i+1})Y^{i+1}dx+\frac{1}{2}\int_{R}\partial_{x}(U^{i}Y^{i})Y^{i+1}dx+\int_{R}|\Lambda^{\frac{\alpha}{2}}Y^{i+1}|^{2}dx=0,
$$
which implies
\begin{align}\label{e4.22}
 \|\Lambda^{\frac{\alpha}{2}}Y^{i+1}\|_{L^{2}}^{2}
 &\leq \frac{1}{2}(|\int_{R}\partial_{x}(U^{i}Y^{i+1})Y^{i+1}dx|+|\int_{R}\partial_{x}(U^{i}Y^{i})Y^{i+1}dx|)  \nonumber\\
 &\leq \frac{1}{2}\int_{R}\Lambda^{1-\frac{\alpha}{2}}(U^{i}Y^{i+1})\Lambda^{\frac{\alpha}{2}}Y^{i+1}dx+\frac{1}{2}\int_{R}\Lambda^{1-\frac{\alpha}{2}}(U^{i}Y^{i})
 \Lambda^{\frac{\alpha}{2}}Y^{i+1}dx \nonumber\\
 &\leq \frac{1}{2}\|\Lambda^{1-\frac{\alpha}{2}}(U^{i}Y^{i+1})\|_{L^{2}}\|\Lambda^{\frac{\alpha}{2}}Y^{i+1}\|_{L^{2}}+\frac{1}{2}
 \|\Lambda^{1-\frac{\alpha}{2}}(U^{i}Y^{i})\|_{L^{2}}\|\Lambda^{\frac{\alpha}{2}}Y^{i+1}\|_{L^{2}}.
\end{align}
On the other hand, we can get
\begin{align}\label{e4.23}
\|\Lambda^{1-\frac{\alpha}{2}}(U^{i}Y^{i})\|_{L^{2}}
 &\leq \|\Lambda^{1-\frac{\alpha}{2}}U^{i}\|_{L^{2}}\|Y^{i}\|_{L^{\infty}}+\|U^{i}\|_{L^{\frac{1}{\alpha-1}}}
\|\Lambda^{1-\frac{\alpha}{2}}Y^{i}\|_{L^{\frac{2}{3-2\alpha}}}  \nonumber\\
 &\leq (\|\Lambda^{1-\frac{\alpha}{2}}U^{i}\|_{L^{2}}+\|U^{i}\|_{L^{\frac{1}{\alpha-1}}})\|\Lambda^{\frac{\alpha}{2}}Y^{i}\|_{L^{2}} \nonumber\\
 &\leq 2\|U^{i}\|_{H^{\frac{\alpha}{2}}}\|\Lambda^{\frac{\alpha}{2}}Y^{i}\|_{L^{2}},
\end{align}
and
\begin{equation}\label{e4.24}
\|\Lambda^{1-\frac{\alpha}{2}}(U^{i}Y^{i+1})\|_{L^{2}}\leq 2\|U^{i}\|_{H^{\frac{\alpha}{2}}}\|\Lambda^{\frac{\alpha}{2}}Y^{i+1}\|_{L^{2}}.
\end{equation}
Choosing $C(\alpha,\epsilon)$ and $\|f\|_{X}\leq C(\alpha,\epsilon)$, such that
 $C\epsilon^{-1}\|f\|_{X}\leq \frac{1}{3}$, then we have $\|U^{i}\|_{H^{\frac{\alpha}{2}}}\leq C\epsilon^{-1}\|f\|_{X}\leq \frac{1}{3}$. Combining $\|U^{i}\|_{H^{\frac{\alpha}{2}}}\leq \frac{1}{3}$ with (\ref{e4.22})-(\ref{e4.24}), we obtain
\begin{equation}\label{e4.25}
\|\Lambda^{\frac{\alpha}{2}}Y^{i+1}\|_{L^{2}}\leq \frac{1}{2}\|\Lambda^{\frac{\alpha}{2}}Y^{i}\|_{L^{2}}.
\end{equation}
Applying the estimate (\ref{e4.25}) recursively gives
$$
\|\Lambda^{\frac{\alpha}{2}}Y^{i+1}\|_{L^{2}}\leq (\frac{1}{2})^{i}\|\Lambda^{\frac{\alpha}{2}}Y^{1}\|_{L^{2}},
$$
then $Y^{i}\rightarrow 0$ in $\dot{H}^{\frac{\alpha}{2}}(R)$ as $i\rightarrow \infty$. Hence the sequence $\{U^{i}\}$ is Cauchy sequence in $\dot{H}^{\frac{\alpha}{2}}$ and there exists a limit $U\in\dot{H}^{\frac{\alpha}{2}}(R)$ such that
\begin{equation}\label{e4.26}
U^{i}\rightarrow U \quad in ~~~~\dot{H}^{\frac{\alpha}{2}}(R).
\end{equation}
In addition, it follows from Lemma \ref{lem4.5} that $\|U\|_{H^{\frac{\alpha}{2}}}\leq C\epsilon^{-1}\|f\|_{X}$.

Now we are in the position to show that $U$ is a unique solution to (\ref{e12}). By Lemma \ref{lem4.1}, we have
$$
-\frac{1}{2}\langle U^{i}U^{i+1}, \psi_{x}\rangle+\langle \Lambda^{\frac{\alpha}{2}}U^{i+1}, \Lambda^{\frac{\alpha}{2}}\psi\rangle=\langle f, \psi\rangle, \quad  \forall\psi\in C^{\infty}_{0}(R).
$$
It follows from (\ref{e4.26}) immediately that
\begin{equation}\label{e4.27}
\langle \Lambda^{\frac{\alpha}{2}}U^{i+1}, \Lambda^{\frac{\alpha}{2}}\psi\rangle\rightarrow \langle \Lambda^{\frac{\alpha}{2}}U, \Lambda^{\frac{\alpha}{2}}\psi\rangle.
\end{equation}
For the nonlinear term, we have
\begin{align}\label{e4.28}
~&
\langle U^{i}U^{i+1}, \psi_{x}\rangle-\langle U^{2}, \psi_{x}\rangle \nonumber\\
 &= \langle (U^{i}-U)U^{i+1}+(U^{i+1}-U)U, \psi_{x}\rangle \nonumber\\
 &\leq \|U^{i}-U\|_{L^{\infty}}\|U^{i+1}\|_{L^{2}}\|\psi_{x}\|_{L^{2}}+\|U^{i+1}-U\|_{L^{\infty}}\|U\|_{L^{2}}\|\psi_{x}\|_{L^{2}} \nonumber\\
 &\leq \|U^{i}-U\|_{\dot{H}^{\frac{\alpha}{2}}}\|U^{i+1}\|_{L^{2}}\|\psi_{x}\|_{L^{2}}+\|U^{i+1}-U\|_{\dot{H}^{\frac{\alpha}{2}}}\|U\|_{L^{2}}\|\psi_{x}\|_{L^{2}}\nonumber\\
 &\rightarrow 0,
\end{align}
where (\ref{e4.26}) and the boundedness of $U$ and $U^{i+1}$ are used. Thus, (\ref{e4.27}) and (\ref{e4.28}) imply that $U$ is the weak solution of (\ref{e12}).

Finally, we show that this solution $U$ is unique among all solutions satisfying
$$
\|U\|_{H^{\frac{\alpha}{2}}}\leq C\epsilon^{-1}\|f\|_{X}.
$$
Assume $\widetilde{U}$ is another weak solution of (\ref{e12}) with $\|\widetilde{U}\|_{H^{\frac{\alpha}{2}}}\leq C\epsilon^{-1}\|f\|_{X}$, then the difference $Y=U-\widetilde{U}$ satisfies
\begin{equation}\label{e4.29}
\partial_{x}(UY)+\Lambda^{\alpha}Y=0
\end{equation}
Applying the similar procedure above on (\ref{e4.29}), we can get
$$
\|\Lambda^{\frac{\alpha}{2}}Y\|_{L^{2}}\leq \frac{1}{2}\|\Lambda^{\frac{\alpha}{2}}Y\|_{L^{2}},
$$
which means the solution is unique.
\end{proof}

\section{Stability of the steady-state solution}
In this section, we will discuss the stability of the steady-state solution obtained in Theorem \ref{thm1.4}. Since we can
choose $C(\alpha,\epsilon)$ and $\|f\|_{X}\leq C(\alpha,\epsilon)$,
such that
 $C\epsilon^{-1}\|f\|_{X}\leq \frac{1}{3}$. Then we have $\|U^{i}\|_{H^{\frac{\alpha}{2}}}\leq C\epsilon^{-1}\|f\|_{X}\leq \frac{1}{3}$, which would be used in the following.
\subsection{The proof of Theorem \ref{thm1.5}}
\begin{proof}
We assume that (\ref{e1.6}) holds provided the force $f\in \dot{H}^{-\frac{\alpha}{2}}(R)\cap L^{2}(R)$, which would be proved in next subsection.
In the following, we mainly establish
$$
\lim_{t\rightarrow \infty}\|w\|_{L^{2}}=\lim_{t\rightarrow \infty}\|u(t)-U\|_{L^{2}}=0,
$$
where $w$ is the difference between $u(t,x)$ and $U(x)$ solving
\begin{equation}\label{e5.1}
\left\{ \begin{array}{ll}
 w_{t}+ww_{x}+(Uw)_{x}+\Lambda^{\alpha}w=0,  \\
w(0,x)=u(0,x)-U=\theta.
 \end{array} \right.
\end{equation}

Multiplying both sides of (\ref{e5.1}) by $w$ and integrating on $R$ give
$$
\frac{1}{2}\frac{d}{dt}\int_{R}|w|^{2}dx+\int_{R}|\Lambda^{\frac{\alpha}{2}}w|^{2}dx=-\int_{R}(Uw)_{x}wdx
$$
and choosing $ C(\alpha, \varepsilon)$ small such that $\|U\|_{H^{\frac{\alpha}{2}}}\leq \frac{1}{3}$, then
\begin{equation}
\begin{aligned}
|\int_{R}(Uw)_{x}wdx|
 &\leq |\int_{R} \Lambda^{1-\frac{\alpha}{2}}(Uw)\Lambda^{\frac{\alpha}{2}}wdx| \nonumber\\
 &\leq \|\Lambda^{1-\frac{\alpha}{2}}(Uw)\|_{L^{2}} \|\Lambda^{\frac{\alpha}{2}}w\|_{L^{2}}  \nonumber\\
 &\leq (\|\Lambda^{1-\frac{\alpha}{2}}U\|_{L^{2}}\|w\|_{L^{\infty}}+\|U\|_{L^{\frac{1}{\alpha-1}}}
\|\Lambda^{1-\frac{\alpha}{2}}w\|_{L^{\frac{2}{3-2\alpha}}})\|\Lambda^{\frac{\alpha}{2}}w\|_{L^{2}}  \nonumber\\
 &\leq (\|\Lambda^{1-\frac{\alpha}{2}}U\|_{L^{2}}+\|U\|_{L^{\frac{1}{\alpha-1}}})\|\Lambda^{\frac{\alpha}{2}}w\|_{L^{2}}^{2} \nonumber\\
 &\leq 2\|U\|_{H^{\frac{\alpha}{2}}}\|\Lambda^{\frac{\alpha}{2}}w\|_{L^{2}}^{2}\leq \frac{2}{3}\|\Lambda^{\frac{\alpha}{2}}w\|_{L^{2}}^{2}.
\end{aligned}
\end{equation}
Therefore,
$$
\frac{d}{dt}\int_{R}|w|^{2}dx+\frac{2}{3}\int_{R}|\Lambda^{\frac{\alpha}{2}}w|^{2}dx\leq0.
$$
Integrating with respect to time from $0$ to $t$, we have
  \begin{equation}\label{e5.2}
 \int_{R}|w|^{2}dx+\frac{2}{3}\int_{0}^{t}\int_{R}|\Lambda^{\frac{\alpha}{2}}w|^{2}dxds\leq \|\theta\|_{L^{2}}^{2}.
 \end{equation}

 On the other hand, the fact $\|U\|_{L^{\infty}}\leq\|U\|_{H^{\frac{\alpha}{2}}}\leq \frac{1}{3}$ and (\ref{e1.6}) imply that $w=u-U$ is bounded in $L^{\infty}((0,T)\times R)$. At the same time, (\ref{e5.2}) implies that $w$ is also bounded in $L^{\infty}((0,T);L^{2}(R))$. Multiplying both sides of (\ref{e5.1}) by $e^{-2(t-s)\Lambda^{\alpha}}w(s)$ and integrating over the whole space, we have
\begin{equation}
\begin{aligned}
~
&\frac{1}{2}\frac{d}{ds}\int_{R}|e^{-(t-s)\Lambda^{\alpha}} w(s)|^{2}dx
+\frac{1}{2}\int_{R}e^{-(t-s)\Lambda^{\alpha}} (w^{2})_{x}e^{-(t-s)\Lambda^{\alpha}} w(s)dx  \nonumber\\
&+ \int_{R}e^{-(t-s)\Lambda^{\alpha}} (Uw)_{x}e^{-(t-s)\Lambda^{\alpha}} w(s)dx=0.
\end{aligned}
\end{equation}
Integrating over the time interval $[s,t]$ yields
  \begin{equation}
\begin{aligned}
\| w(t)\|_{L^{2}}^{2}
&\leq \|e^{-(t-s)\Lambda^{\alpha}} w(s)\|_{L^{2}}^{2}  \nonumber\\
&+\int_{s}^{t}|\langle (ww)_{x},e^{-2(t-\tau)\Lambda^{\alpha}} w(\tau) \rangle|d\tau +2\int_{s}^{t}|\langle (Uw)_{x},e^{-2(t-\tau)\Lambda^{\alpha}} w(\tau) \rangle|d\tau \nonumber\\
&\leq \|e^{-(t-s)\Lambda^{\alpha}} w(s)\|_{L^{2}}^{2}  \nonumber\\
&+\int_{s}^{t}\|\Lambda^{1-\frac{\alpha}{2}}(ww)\|_{L^{2}}\|\Lambda^{\frac{\alpha}{2}}w\|_{L^{2}} d\tau+2\int_{s}^{t}\|\Lambda^{1-\frac{\alpha}{2}}(Uw)\|_{L^{2}}\|\Lambda^{\frac{\alpha}{2}}w\|_{L^{2}} d\tau  \nonumber\\
&\leq \|e^{-(t-s)\Lambda^{\alpha}} w(s)\|_{L^{2}}^{2} \nonumber\\
&+\int_{s}^{t}(\|\Lambda^{1-\frac{\alpha}{2}}w\|_{L^{2}}\|w\|_{L^{\infty}}+\|w\|_{L^{\frac{1}{\alpha-1}}}\|w\|_{L^{\frac{2}{3-2\alpha}}})\|\Lambda^{\frac{\alpha}{2}}w\|_{L^{2}}d\tau  \nonumber\\
&+2\int_{s}^{t}(\|\Lambda^{1-\frac{\alpha}{2}}U\|_{L^{2}}\|w\|_{L^{\infty}}+\|U\|_{L^{\frac{1}{\alpha-1}}}\|w\|_{L^{\frac{2}{3-2\alpha}}})\|\Lambda^{\frac{\alpha}{2}}w\|_{L^{2}}d\tau.
\end{aligned}
\end{equation}
Sobolev embedding theorem and the interpolation theorem imply
  \begin{equation}
\begin{aligned}
\| w(t)\|_{L^{2}}^{2}
&\leq \|e^{-(t-s)\Lambda^{\alpha}} w(s)\|_{L^{2}}^{2} \nonumber\\
&+\int_{s}^{t}(\|w\|_{L^{2}}+\|\Lambda^{\frac{\alpha}{2}}w\|_{L^{2}})\|w\|_{L^{\infty}}\|\Lambda^{\frac{\alpha}{2}}w\|_{L^{2}}+(\|w\|_{L^{2}}+\|w\|_{L^{\infty}})\|\Lambda^{\frac{\alpha}{2}}w\|_{L^{2}}^{2}d\tau  \nonumber\\
&+2\int_{s}^{t}(\|\Lambda^{1-\frac{\alpha}{2}}U\|_{L^{2}}+\|U\|_{L^{\frac{1}{\alpha-1}}})\|\Lambda^{\frac{\alpha}{2}}w\|_{L^{2}}^{2}d\tau \nonumber\\
&\leq \|e^{-(t-s)\Lambda^{\alpha}} w(s)\|_{L^{2}}^{2}+ C\int_{s}^{t}\|\Lambda^{\frac{\alpha}{2}}w\|_{L^{2}}^{2}d\tau+4\|U\|_{H^{\frac{\alpha}{2}}}\int_{s}^{t}\|\Lambda^{\frac{\alpha}{2}}w\|_{L^{2}}^{2}d\tau \nonumber\\
&\leq \|e^{-(t-s)\Lambda^{\alpha}} w(s)\|_{L^{2}}^{2}+(C+\frac{4}{3}) \int_{s}^{t} \|\Lambda^{\frac{\alpha}{2}}w\|_{L^{2}}^{2} d\tau,
\end{aligned}
\end{equation}
where $C=2(\|w\|_{L^{\infty}((0,T)\times R)}+\|w\|_{L^{\infty}(0,T;L^{2}(R))})$.

Thus there holds that
 $$
 \lim_{t\rightarrow \infty}\| w(t)\|_{L^{2}}^{2}\leq (C+\frac{4}{3}) \int_{s}^{\infty} \|\Lambda^{\frac{\alpha}{2}}w\|_{L^{2}}^{2} d\tau.
 $$
Estimate (\ref{e5.2}) implies that the right-hand side of the above inequality tends to $0$ as $s\rightarrow \infty$. Thus we have
$$
\lim_{t\rightarrow\infty}\|w(t)\|_{L^{2}}^{2}=0.
$$
\end{proof}

\subsection{The regularity of viscosity weak solution $u(t,x)$ to (\ref{e11})}
In this subsection, we improve the regularity of viscosity weak solution $u(t,x)$ to (\ref{e11}). In the case of the surface Quasi-geostrophic equation in $R^{2}$, Cheskidov and Dai have derived a level set energy inequality in \cite{14} and \cite{15}. We use the similar method to sketch the proof for Burgers equation in $R$.
\begin{lem}
Let $u(t,x)$ be a viscosity weak solution to (\ref{e11}) on $[0,T]$ with $u_{0}\in L^{2}$. Then for every $\lambda\in R$ there holds that
\begin{equation}\label{e6.1}
\|u_{\lambda}(t_{2})\|_{L^{2}}^{2}+2\int_{t_{1}}^{t_{2}}\|\Lambda^{\frac{\alpha}{2}}u_{\lambda}\|_{L^{2}}^{2}dt\leq \|u_{\lambda}(t_{1})\|_{L^{2}}^{2}+2\int^{t_{2}}_{t_{1}}\int_{R}fu_{\lambda}dxdt,
\end{equation}
where $0\leq t_{1}\leq t_{2}\leq T$ and $u_{\lambda}=(u-\lambda)_{+}$ or $u_{\lambda}=(u+\lambda)_{-}$.
\end{lem}
\begin{proof}
Denote $h(u)=(u-\lambda)_{+}$, it's easy to see that $h$ is Lipschitz and
$$
h'(u)h(u)=h(u).
$$
Multiplying the equation (\ref{e11}) by $h'(u)h(u)=h(u)$ and integrating over $R$ yield
\begin{equation}\label{e6.2}
\frac{1}{2}\frac{d}{dt}\int_{R} h^{2}(u)dx+\int_{R}uu_{x}h'(u)h(u)dx+\int_{R}\Lambda^{\alpha}u h(u)dx=\int_{R}fh(u)dx.
\end{equation}
Note that on the set $\{x\in R: u(x)\leq\lambda\}$, there holds
$$
uu_{x}h'(u)h(u)=0 \quad and \quad \Lambda^{\alpha}u h(u)=0.
$$
On the set $\{x\in R: u(x)>\lambda\}$, we have
$$
uu_{x}h'(u)h(u)=uu_{x}(u-\lambda)
$$
and
\begin{equation}
\begin{aligned}
\Lambda^{\alpha}u(x)
&=C ~~P.V.\int_{R}\frac{u(x)-u(y)}{|x-y|^{1+\alpha}}dy=C ~~P.V.\int_{R}\frac{u(x)-\lambda-(u(y)-\lambda)}{|x-y|^{1+\alpha}}dy  \nonumber\\
&\geq C ~~P.V.\int_{R}\frac{(u(x)-\lambda)_{+}-(u(y)-\lambda)_{+}}{|x-y|^{1+\alpha}}dy=\Lambda^{\alpha}h(u(x)).
\end{aligned}
\end{equation}
Hence
\begin{equation}\label{e6.3}
\int_{R}uu_{x}h'(u)h(u)dx=0 \quad and \quad \int_{R}\Lambda^{\alpha}u h(u)dx\geq\int_{R}|\Lambda^{\frac{\alpha}{2}}h(u)|^{2}dx.
\end{equation}
(\ref{e6.2}) and (\ref{e6.3}) yield the truncated energy inequality (\ref{e6.1}). For the case of $h(u)=(u+\lambda)_{-}$, the similar process can be carried out.
\end{proof}
Now we use similar idea of \cite{13} to prove the following lemma.
\begin{lem}
Assume that $u(t,x)$ is a viscosity solution of (\ref{e11}) on $[0,\infty)$ with $u_{0}\in L^{2}$ and $f\in \dot{H}^{-\frac{\alpha}{2}}(R)\cap L^{2}(R)$. Then
$$
u\in L^{\infty}((\varepsilon,\infty)\times R)
$$
for every $\varepsilon>0$.
\end{lem}
\begin{proof}
Let
$$
\lambda_{n}=M(1-2^{-n}),
$$
where $M$ will be determined later and denote the truncated function by
$$
u_{\lambda_{n}}=(u-\lambda_{n})_{+}.
$$
By using the De Giogi's technology, for a fixed $t_{0}>0$, we have
$$
u(t_{0},x)\leq M, \quad \forall x\in R
$$
provided establishing
\begin{equation}\label{e6.4}
E_{n}:=\sup_{T_{n}\leq t\leq t_{0}}\|u_{\lambda_{n}}(t)\|_{L^{2}}^{2}+2\int^{t_{0}}_{T_{n}}\|\Lambda^{\frac{\alpha}{2}}u_{\lambda_{n}}(t)\|_{L^{2}}^{2}dt\rightarrow 0,
\end{equation}
as $n\rightarrow \infty$, where $T_{n}:=t_{0}(1-2^{^{-n}})$.

Taking $u_{\lambda}=u_{\lambda_{n}}$ and $t_{1}=t'\in(T_{n-1},T_{n})$, $t_{2}=t>T_{n}$ in (\ref{e6.1}), then taking $t_{1}=t'$, $t_{2}=t_{0}$, adding two inequalities and taking sup in $t$ give
$$
E_{n}\leq 2\|u_{\lambda_{n}}(t')\|_{L^{2}}^{2}+4\int_{T_{n-1}}^{t_{0}}\int_{R}|f(x)u_{\lambda_{n}}(s,x)|dxds
$$
for $t'\in (T_{n-1},T_{n})$. Taking the average in $t'$ on $[T_{n-1},T_{n}]$ yields
\begin{equation}\label{e6.5}
E_{n}\leq \frac{2^{n+1}}{t_{0}}\int_{T_{n-1}}^{t_{0}}\int_{R}u_{\lambda_{n}}^{2}(t')dxdt'+4\int_{T_{n-1}}^{t_{0}}\int_{R}|f(x)u_{\lambda_{n}}(t,x)|dxdt.
\end{equation}
Interpolating between $L^{\infty}(L^{2})$ and $L^{2}(\dot{H}^{\frac{\alpha}{2}})$, we have
\begin{equation}\label{e6.6}
\|u_{\lambda_{n}}\|_{L^{4}([T_{n},t_{0})\times R)}^{4}\leq CE_{n}^{2},
\end{equation}
where $C$ is independent of $n$.
Note that on the set $\{(t,x): u_{\lambda_{n}}(t,x)>0\}$, we have
$$
u>\lambda_{n}>\lambda_{n-1}.
$$
Hence
$$
u_{\lambda_{n-1}}=u-\lambda_{n-1}>\lambda_{n}-\lambda_{n-1}=M2^{-n}.
$$
This implies
\begin{equation}\label{e6.7}
1_{\{u_{\lambda_{n}}>0\}}\leq \frac{2^{2n}}{M^{2}}u_{\lambda_{n-1}}^{2}.
\end{equation}
The fact $u_{\lambda_{n}}\leq u_{\lambda_{n-1}}$ and (\ref{e6.6}) (\ref{e6.7}) give
\begin{align}\label{e6.8}
~
 &\frac{2^{n+1}}{t_{0}}\int^{t_{0}}_{T_{n-1}}\int_{R} u_{\lambda_{n}}^{2}(t',x)dxdt'  \nonumber\\
 &\leq \frac{2^{n+1}}{t_{0}}\int^{t_{0}}_{T_{n-1}}\int_{R} u_{\lambda_{n-1}}^{2}(t',x)1_{\{u_{\lambda_{n}>0}\}}dxdt' \nonumber\\
 &\leq \frac{2^{3n+1}}{t_{0}M^{2}}\int^{t_{0}}_{T_{n-1}}\int_{R} u_{\lambda_{n-1}}^{4}(t',x)dxdt' \nonumber\\
 &\leq \frac{C2^{3n+1}}{ t_{0}M^{2}}E_{n-1}^{2}
\end{align}
In addition, for $f\in L^{2}(R)$, we have
\begin{align}\label{e6.9}
~
 &\int^{t_{0}}_{T_{n-1}}\int_{R} |f(x)u_{\lambda_{n}}(t,x)|dxdt  \nonumber\\
 &\leq \|f\|_{L^{2}}\int^{t_{0}}_{T_{n-1}}(\int_{R} |u_{\lambda_{n}}|^{2}dx)^{\frac{1}{2}}dt \nonumber\\
 &\leq \|f\|_{L^{2}}\int^{t_{0}}_{T_{n-1}}(\int_{R} |u_{\lambda_{n-1}}|^{2}1^{2}_{\{u_{\lambda_{n}>0}\}}dx)^{\frac{1}{2}}dt   \nonumber\\
 &\leq \frac{2^{4n}\|f\|_{L^{2}}}{M^{4}}\int^{t_{0}}_{T_{n-1}}(\int_{R} |u_{\lambda_{n-1}}|^{4}|u_{\lambda_{n-1}}|^{2}dx)^{\frac{1}{2}}dt  \nonumber\\
 &\leq \frac{2^{4n}\|f\|_{L^{2}}}{M^{4}}\int^{t_{0}}_{T_{n-1}}\|u_{\lambda_{n-1}}(t)\|_{L^{\infty}}^{2}(\int_{R}|u_{\lambda_{n-1}}(t)|^{2}dx)^{\frac{1}{2}}dt \nonumber\\
 &\leq \frac{2^{4n}\|f\|_{L^{2}}}{M^{4}}\sup_{T_{n-1}\leq t\leq t_{0}}\|u_{\lambda_{n-1}}(t)\|_{L^{2}}\int^{t_{0}}_{T_{n-1}}\|u_{\lambda_{n-1}}(t)\|_{L^{\infty}}^{2}dt \nonumber\\
 &\leq \frac{C2^{4n-1}\|f\|_{L^{2}}}{ M^{4}} E_{n-1}^{\frac{3}{2}}.
\end{align}
Therefore, (\ref{e6.5}), (\ref{e6.8}) and (\ref{e6.9}) yield
\begin{equation}\label{e6.10}
E_{n}\leq \frac{C2^{3n+1}}{ t_{0}M^{2}}E_{n-1}^{2}+\frac{C2^{4n+1}\|f\|_{L^{2}}}{ M^{4}} E_{n-1}^{\frac{3}{2}},
\end{equation}
where the constant $C$ is independent of $n$. For a large enough
\begin{equation}\label{e6.11}
M \sim \frac{E_{0}^{\frac{1}{2}}}{( t_{0})^{\frac{1}{2}}}+\|f\|_{L^{2}}^{\frac{1}{4}}E_{0}^{\frac{1}{8}},
\end{equation}
the nonlinear iteration inequality (\ref{e6.10}) implies that $E_{n}$ converges to $0$ as $n\rightarrow \infty$. Then $u(t_{0},x)\leq M$ for almost every $x$. The same argument can be applied to $u_{\lambda_{n}}=(u+\lambda_{n})_{-}$.

On the other hand, it follows from the energy inequality that
\begin{equation}
 \begin{aligned}
\|u(t_{0})\|_{L^{2}}^{2}+2\int_{0}^{t_{0}}\|\Lambda^{\frac{\alpha}{2}}u(t)\|_{L^{2}}^{2}dt
 &\leq \|u_{0}\|_{L^{2}}^{2}+2\int^{t_{0}}_{0}\int_{R}f(x)u(t,x)dxdt \nonumber\\
 &\leq \|u_{0}\|_{L^{2}}^{2}+\frac{C}{}\int^{t_{0}}_{0}\|f\|_{\dot{H}^{-\frac{\alpha}{2}}}^{2}dt+
 \int^{t_{0}}_{0}\|u\|_{\dot{H}^{\frac{\alpha}{2}}}^{2}dt,
 \end{aligned}
 \end{equation}
 which means
 \begin{equation}\label{e6.12}
E_{0}\leq 2\|u_{0}\|_{L^{2}}^{2}+C\|f\|_{\dot{H}^{-\frac{\alpha}{2}}}^{2}t_{0}.
\end{equation}
(\ref{e6.11}), (\ref{e6.12}) and  Young's inequality imply
 \begin{align}\label{e6.13}
\|u(t)\|_{L^{\infty}}
 &\lesssim E_{0}^{\frac{1}{2}}( t)^{-\frac{1}{2}}+[\|f\|_{L^{2}}^{\frac{1}{3}}( t)^{\frac{1}{6}}]^{\frac{3}{4}}[E_{0}^{\frac{1}{2}}( t)^{-\frac{1}{2}}]^{\frac{1}{4}} \nonumber\\
 &\lesssim E_{0}^{\frac{1}{2}}( t)^{-\frac{1}{2}}+\|f\|_{L^{2}}^{\frac{1}{3}}( t)^{\frac{1}{6}} \nonumber\\
 &\lesssim\|u_{0}\|_{L^{2}}(t)^{-\frac{1}{2}}+\|f\|_{\dot{H}^{-\frac{\alpha}{2}}}+\|f\|_{L^{2}}^{\frac{1}{3}}t^{\frac{1}{6}} \nonumber\\
 &\lesssim \|u_{0}\|_{L^{2}}(t)^{-\frac{1}{2}}+\|f\|_{\dot{H}^{-\frac{\alpha}{2}}}+\|f\|_{L^{2}}^{\frac{1}{3}}, \quad for~~t\leq 1.
 \end{align}
To show that this bound also holds for $t>1$, fixing $t=\tau=1$ in (\ref{e6.13}) and shifting it by $t-\tau$ in time, thus we obtain
 \begin{equation}\label{e6.14}
\|u(t)\|_{L^{\infty}}\lesssim\frac{ \|u(t-\tau)\|_{L^{2}}}{( t)^{\frac{1}{2}}}+\frac{\|f\|_{\dot{H}^{-\frac{\alpha}{2}}}}{}+\|f\|_{L^{2}}^{\frac{1}{3}}, \quad t\geq \tau.
\end{equation}
Due to the energy inequality in Corollary \ref{cor3.4}, we have
 \begin{equation}\label{e6.15}
\|u(t-\tau)\|_{L^{2}}\leq \|u_{0}\|_{L^{2}}+(\frac{t-\tau}{2})^{\frac{1}{2}}\|f\|_{\dot{H}^{-\frac{\alpha}{2}}}.
\end{equation}
Combining this with (\ref{e6.14}), we get
$$
\|u(t)\|_{L^{\infty}}\lesssim \frac{\|u_{0}\|_{L^{2}}}{( t)^{\frac{1}{2}}}+\|f\|_{\dot{H}^{-\frac{\alpha}{2}}}+\|f\|_{\dot{H}^{-\frac{\alpha}{2}}}+\|f\|_{L^{2}}^{\frac{1}{3}}, \quad t\geq \tau.
$$
\end{proof}

\section*{Acknowledgments}
This project is supported by National Natural Science Foundation of China (No:11571057).

\end{document}